\documentclass[12pt,a4paper]{amsart}

\usepackage{fullpage,etex,amsfonts,amsthm,graphicx,bm}

\usepackage{amsmath,amssymb,verbatim}
\usepackage{color}
\usepackage[T1]{fontenc}
\usepackage[textsize=footnotesize]{todonotes}

\usepackage{pgf}
\usepackage{tikz}

\usepackage{amsaddr}

\usetikzlibrary{arrows,automata}

\theoremstyle{plain}
\newtheorem{proposition}{Proposition}
\newtheorem{corollary}[proposition]{Corollary}
\newtheorem{theorem}[proposition]{Theorem}
\newtheorem{lemma}[proposition]{Lemma}
\newtheorem{definition}[proposition]{Definition}
\theoremstyle{definition}
\newtheorem{example}[proposition]{Example}
\newtheorem{remark}[proposition]{Remark}
\theoremstyle{remark}
\newtheorem{question}{Question}

\title{Modal operators and toric ideals}

\author{Riccardo Camerlo}
\address{Dipartimento di matematica, Universit\`a di Genova, Via Dodecaneso 35, 16146 Genova --- Italy}
\email{camerlo@dima.unige.it}
\author{Giovanni Pistone}
\address{de Castro Statistics, Collegio Carlo Alberto, Piazza Vincenzo Arbarello 8, 10122 Torino --- Italy}
\email{giovanni.pistone@carloalberto.org}
\author{Fabio Rapallo}
\address{Dipartimento di scienze e innovazione tecnologica, Universit\`a del Piemonte Orientale, Viale Teresa Michel 11, 15121 Alessandria --- Italy}
\email{fabio.rapallo@uniupo.it}


\date{}

\DeclareMathOperator{\Kripke}{\mathcal K}
\DeclareMathOperator{\Fun}{\mathcal F}
\DeclareMathOperator{\Ideal}{Ideal}

\DeclareMathOperator{\Elim}{Elim}
\DeclareMathOperator{\Ker}{Ker}
\DeclareMathOperator{\Necessarly}{\Box}
\DeclareMathOperator{\Possibly}{\lozenge}

\DeclareMathOperator{\pr}{pr}
\DeclareMathOperator{\Prop}{\mathfrak P}

\newcommand{\setof}[2]{\left\{#1 \middle| #2 \right\}}
\newcommand{\set}[1]{\left\{#1\right\}}

\newcommand{\elimof}[2]{\Elim\left(#1,#2\right)}
\newcommand{\kerof}[1]{\Ker\left(#1\right)}

\newcommand{\N}{\mathbb{N}}
\newcommand{\Z}{\mathbb{Z}}

\newcommand{\C}{\mathbb{C}}

\makeatletter
\newif\if@borderstar
\def\bordermatrix{\@ifnextchar*{%
\@borderstartrue\@bordermatrix@i}{\@borderstarfalse\@bordermatrix@i*}%
}
\def\@bordermatrix@i*{\@ifnextchar[{\@bordermatrix@ii}{\@bordermatrix@ii[() ]}}
\def\@bordermatrix@ii[#1]#2{%
\begingroup
\m@th\@tempdima8.75\p@\setbox\z@\vbox{%
\def\cr{\crcr\noalign{\kern 2\p@\global\let\cr\endline }}%
\ialign {$##$\hfil\kern 2\p@\kern\@tempdima & \thinspace %
\hfil $##$\hfil && \quad\hfil $##$\hfil\crcr\omit\strut %
\hfil\crcr\noalign{\kern -\baselineskip}#2\crcr\omit %
\strut\cr}}%
\setbox\tw@\vbox{\unvcopy\z@\global\setbox\@ne\lastbox}%
\setbox\tw@\hbox{\unhbox\@ne\unskip\global\setbox\@ne\lastbox}%
\setbox\tw@\hbox{%
$\kern\wd\@ne\kern -\@tempdima\left\@firstoftwo#1%
\if@borderstar\kern2pt\else\kern -\wd\@ne\fi%
\global\setbox\@ne\vbox{\box\@ne\if@borderstar\else\kern 2\p@\fi}%
\vcenter{\if@borderstar\else\kern -\ht\@ne\fi%
\unvbox\z@\kern-\if@borderstar2\fi\baselineskip}%
\if@borderstar\kern-2\@tempdima\kern2\p@\else\,\fi\right\@secondoftwo#1 $%
}\null \;\vbox{\kern\ht\@ne\box\tw@}%
\endgroup}
\makeatother

\usepackage{tikz}
\usepackage{tikz,fullpage}
\usetikzlibrary{arrows,%
                petri,%
                topaths}%
\usepackage{tkz-berge}
\usepackage[position=top]{subfig}

\begin{document}

\begin{abstract}
In the present paper we consider modal propositional logic and look for the constraints that are imposed to the propositions of the special type $\Necessarly a$ by the structure of the relevant finite Kripke frame. We translate the usual language of modal propositional logic in terms of notions of commutative algebra, namely polynomial rings, ideals, and bases of ideals. We use extensively the perspective obtained in previous works in Algebraic Statistics. We prove that the constraints on $\Necessarly a$ can be derived through a binomial ideal containing a toric ideal and we give sufficient conditions under which the toric ideal, together with the fact that the truth values are in $ \set{0,1} $, fully describes the constraints.
\end{abstract}

\maketitle

\section{Introduction}

Propositional Modal Logic extends propositional logic by adding two operators, $\Necessarly$ and $\Possibly$.
Given a proposition $p$, one can form the propositions:
\begin{itemize}
\item[ ] $\Necessarly p$, which can be read ``necessarily $p$''; and
\item[ ] $\Possibly p$, which can be read ``possibly $p$''.
\end{itemize}
One of the two operators can be taken as primitive and the other as defined, setting
\begin{equation} \label{duality}
\Possibly p = \neg \Necessarly \neg p
\qquad \text{or} \qquad
 \Necessarly p = \neg \Possibly \neg p \ .
\end{equation}
We refer the reader to \cite{blackburn} as a basic text in modal logic.

S.A. Kripke \cite{kripke:1963} has provided a semantics for modal logic consisting in fixing a set of \emph{possible worlds} and a binary relation specifying which worlds $w'$ are \emph{accessible} from a given world $w$.

Let $ \Prop $ be the set of modal formulas, built starting with a given set of propositional variables.

\begin{definition} \label{defkripke}
\begin{itemize}
\item A \emph{Kripke frame} is a pair $\Kripke = (W,\mathcal E)$ where $W$, called the \emph{universe} of $ \Kripke $, is a non-empty set of \emph{worlds} and $\mathcal E$ is a binary relation on $W$.
\item A Kripke frame $ \Kripke =(W, \mathcal E )$ is \emph{locally finite} if for every $w\in W$ the set
\[
\setof{w'\in W}{(w,w')\in \mathcal E }
\]
is finite.
\item A \emph{subframe} of the Kripke frame $ \Kripke =(W, \mathcal E )$ is a Kripke frame $ \Kripke'=(W', \mathcal E')$ such that $W'\subseteq W, \mathcal E'= \mathcal E \cap (W'\times W')$.
\end{itemize}
\end{definition}

\begin{definition}
A \emph{Kripke model} $ \Kripke_{\Phi }=(W, \mathcal E ,\Phi )$ is a Kripke frame $ \Kripke =(W, \mathcal E )$ endowed with a function $\Phi $ from $ \Prop \times W$ to the Boolean algebra $\{ 0,1\} $, assigning a truth value $\Phi (p,w)$ --- that can be either $0$ (false) or $1$ (true) --- at each world $w$ for each proposition $p$.
Such an assignment must satisfy the following conditions:
\begin{itemize}
\item[($\neg $)] $\Phi (\neg p,w)=1-\Phi (p,w)$
\item[($\wedge$)] $\Phi (p\wedge q,w)=\Phi (p,w)\Phi (q,w)$
\item[($ \Necessarly $)] $\Phi ( \Necessarly p,w)=\prod_{(w,w')\in \mathcal E }\Phi (p,w')$
\end{itemize}
From these equations one can recover the conditions on the remaining logical symbols $\vee, \rightarrow ,\leftrightarrow , \Possibly $ (see also Tab. \ref{table:logical-operator}).
\end{definition}
\begin{table}
\begin{equation*}
\begin{array}[t]{|c|c|}
\hline
\neg a & 1-a \\ a \wedge b & ab \\ a \vee b & a + b - ab \\ a \rightarrow b & 1 - a +ab \\ a \leftrightarrow b & 1 -a - b + 2ab \\
\hline
\end{array}
\end{equation*}
\caption{\label{table:logical-operator} Algebraic translation of logical operators}
\end{table}

In the notation of Mathematical Logic, the equation $\Phi (p,w)=1$ is often written $( \Kripke_{\Phi },w)\Vdash p$.
If $( \Kripke_{\Phi },w)\Vdash p$ for all $w\in W$, then one can write $ \Kripke_{\Phi }\Vdash p$; if $ \Kripke_{\Phi }\Vdash p$ for all possible $\Phi $, this is denoted $ \Kripke \Vdash p$.

Notice that any fixed Kripke model $ \Kripke_{\Phi }$ determines, for every $p\in \Prop $, a function $W\to\{ 0,1\} $ that takes value $1$ if and only if $( \Kripke_{\Phi },w)\Vdash p$: this is the characteristic function --- or indicator function in the probabilistic and statistical literature, where {\it characteristic function} has a different meaning --- of the truth set of $p$.
With a slightly abusive notation, when the Kripke model $ \Kripke_{\Phi }$ is understood, one can denote such a function with the same symbol $p$ as the proposition.
For instance, the proposition $\Necessarly p$ is true in a world $w$ if, and only if, $p$ is true in any world $w'$ which is accessible from $w$.
So, its truth value is
\begin{equation}\label{eq:def}
  \Necessarly p(w) = \prod_{(w,w') \in \mathcal E} p(w')\ .
\end{equation}
Notice also that different propositions $p,q$ may give rise to the same function $W\to\{ 0,1\} $: this happens exactly when $ \Kripke_{\Phi }\Vdash p\leftrightarrow q$.

Eq.~\eqref{eq:def} defines $ \Necessarly $ as a function (in fact, a morphism --- see Proposition \ref{boxmorphism} below) from the monoid $\{ 0,1\}^W$ (endowed with the operation of pointwise multiplication) to itself.
As the function $ \Necessarly $ depends on the Kripke frame $ \Kripke $, it should be denoted by $ \Necessarly_{ \Kripke }$; however, we drop subscript $ \Kripke $ unless there is more than one Kripke model at stake.

Since the elements of $\{ 0,1\}^W$ are the characteristic functions of subsets of $W$, function $ \Necessarly $ can also be viewed as a function of the monoid $ \mathcal P (W)$, the powerset of $W$ endowed with the operation of intersection, into itself, defined by
\[
w\in \Necessarly A\Leftrightarrow\forall w'\in W\ (w \mathcal E w'\Rightarrow w'\in A)\ .
\]

In the present paper we discuss some properties of the operator $ \Necessarly $, with special reference to the tools of Polynomial Commutative Algebra that are used in Algebraic Statistics. See \cite{pistone|riccomagno|wynn:2001} for a general reference. Such an approach is suggested by the very form of Eq.~\eqref{eq:def}.

In fact, the probability distribution of a random variable $X$ with values in a finite set with $K$ points is given by an (unnormalized) probability function that is, a vector of non-negative real numbers $q=(q_1, \ldots , q_K)$. The (normalized) probability function is obtained by dividing by $\sum_k q_k$. In many classical statistical applications e.g., in the framework of multivariate categorical data, a class of statistical models which is widely used to conveniently parametrize probability distributions is the class of log-linear models, see e.g. the monograph~\cite{Fienberg}. Assuming strictly positive probabilities, an important case of log-linear model is
\begin{equation}\label{loglin}
\log(q) = A\theta
\end{equation}
where $A$ is a $K \times H$ integer valued matrix, the so-called model matrix, and $\theta$ is a vector of parameters of length $H$. Exponentiating Eq.~\eqref{loglin}, one obtains
\begin{equation}\label{expon}
q=\zeta^A \qquad \mbox{i.e.} \qquad q_k= \prod_h \zeta_h^{A_{k,h}}, \qquad k=1, \ldots, K
\end{equation}
where $\zeta=(\zeta_1, \ldots, \zeta_H)=(\exp(\theta_1),\ldots,\exp(\theta_H))$ is a vector of new positive parameters. Notice that in Eq.~\eqref{expon} the strict positivity of the $q_k$'s is not required and hence the latter equation defines an extension of the former. In particular, Eq.~\eqref{expon} makes sense for 0-1-valued $q$'s and for a model matrix equal to an adjacency matrix of a graph.

The monomial expression in Eq.~\eqref{expon} can be used to derive the implicit equations on the $q$'s implied by the model, by eliminating the $\zeta$ variables from the polynomial system. It is a well known fact in polynomial algebra that the elimination of the $\zeta$ variables in Eq.~\eqref{expon} leads to a special class of ideals, namely the toric ideal ${\mathcal I}_A$ associated to the matrix $A$. In fact, toric ideals, which in turn are a special kind of binomial ideals, are among the most prominent ingredients of Algebraic Statistics.
For this theory, we refer to \cite[Ch. 4]{sturmfels:1996}.

Although the variables of the toric ideal are actually the probabilities $q_1, \ldots , q_k$, nevertheless the ideal ${\mathcal I}_A$ can be considered in the polynomial ring ${\mathbb C}[q_1, \ldots , q_k]$.
The central observation here is that Eq.~\eqref{expon} has the same monomial structure as the operator $ \Necessarly $ in Eq.~\eqref{eq:def}. Thus, the idea we want to develop in this paper is that some techniques from Algebraic Statistics can be applied to operators coming from Modal Logic.

\subsection{Overview of the paper}
In section \ref{sec:kripkeframe} we characterize when $ \Necessarly $ is an isomorphism of the monoid $\{ 0,1\}^W$ (Theorem \ref{mainthm}).
Notice that $ \mathrm{range} ( \Necessarly )$, the range of operator $ \Necessarly$, being contained in $\{ 0,1\}^W$ can be viewed as a subset of the affine space $ \C^K$ in the case the Kripke frame is finite and has cardinality $K$.
Thus, applying the aforementioned arguments, in section \ref{sec:commutative-algebra} we describe an algebraic method to obtain binomial equations for $ \mathrm{range} ( \Necessarly )$ as a subvariety of $ \C^K$ (Theorem \ref{mainproposition}).

In the general case, the set of such equations can be fairly complicated.
In contrast, we give in Definition \ref{deftame} a notion of \emph{tameness} that amounts to a substantial simplification of this set of equations; in Definition \ref{defcut} we isolate an interesting class of tame Kripke frames.
We also study some algebraic properties of this notion of tameness.

In section \ref{examples} we deal with examples: with them we demonstrate how the approach through binomial ideals can be implemented using symbolic software for the explicit determination of the equations for $ \mathrm{range} ( \Necessarly )$.
Finally, in section \ref{fraq} we discuss the results obtained in the paper and offer a list of questions that remain unanswered and that could lead to further research in the subject.

\section{Operator $ \Necessarly $ as a morphism} \label{boxmorph}
\label{sec:kripkeframe}
We denote by $ \Fun (W)$ the set of all complex-valued functions on $W$.
So $\{ 0,1\}^W$ is a subset of $ \Fun (W)$, namely it is the set of those functions $a$ such that $a^2=a$.

Given a Kripke frame $ \Kripke =(W, \mathcal E )$, the \emph{adjacency matrix} of $ \Kripke $ is the matrix $E:W\times W\to \set{0,1} $ such that $w \mathcal E w'$ if, and only if, $E(w,w')=1$. Each $w \in W$ has a set of neighbors $N(w) = \setof{w' \in W}{E(w,w')=1}$: we call this set the \emph{neighborhood} of $w$.

Eq.~\eqref{eq:def}, together with Eq.~\eqref{duality}, defines modal operators on $ \set{0,1}^W$.
However, when $ \Kripke $ is locally finite, such a definition extends to the entire set of functions $ \Fun (W)$.

\begin{definition}[Modal operators on complex-valued functions] \label{def:modal}
If $ \Kripke $ is locally finite, we define the operators $\Necessarly \colon \Fun (W) \to \Fun (W)$ and $\Possibly \colon \Fun (W) \to \Fun (W) $ by
  \begin{equation} \label{eq:square}
   \Necessarly a(w)= \prod_{w' \in N(w)} a(w') \qquad
\text{and} \qquad
    \Possibly a(w)=1- \Necessarly (1-a)(w)\ .
  \end{equation}
  \end{definition}

  Consider the adjacency matrix $E$ of the Kripke frame. We can write \eqref{eq:square} as
  \begin{equation}\label{eq:modalistoric}
    \Necessarly a(w) = \prod_{w' \in W} a(w')^{E(w,w')} \ .
  \end{equation}

\begin{proposition} \label{boxmorphism}
  The modal operator $\Necessarly$ is a homomorphism of the multiplicative monoid $ \set{0,1}^W$.
If $ \Kripke $ is locally finite, then it extends to a homomorphism of the multiplicative monoid $ \Fun (W) $.
\end{proposition}

\begin{proof}
In fact, $\Necessarly 1 = 1$ and
\begin{multline*}
  \Necessarly (a\cdot b)(w) = \prod_{w'\in W} (a(w')b(w'))^{E(w,w')} = \\ \prod_{w'\in W} a(w')^{E(w,w')}\prod_{w'\in W} b(w')^{E(w,w')} = (\Necessarly a\cdot \Necessarly b) (w) \ .
\end{multline*}
\end{proof}

\begin{remark} \label{movedrmk}
Notice that in $ \set{0,1}^W$ the product operation $\cdot $ can be extended to an infinitary operation $\bigwedge $, setting for any $I\subseteq \set{0,1}^W$ and any $w\in W$,
\[
\bigwedge_{a\in I}a(w)=\min \set{a(w)}_{w\in W}.
\]
Then $ \Necessarly : \set{0,1}^W\to \set{0,1}^W$ preserves this operation, as
\begin{multline} \label{neweqn}
\Necessarly (\bigwedge_{a\in I}a)(w)= \min \setof{\bigwedge_{a\in I}a(w')}{w \mathcal E w'}=\min \setof{a(w')}{a\in I,w \mathcal E w'}= \\
\min \set{ \Necessarly a(w)}_{a\in I}=\bigwedge_{a\in I} \Necessarly a(w).
\end{multline}

In fact, $ \set{0,1}^W$ carries also a partial order $\leq $ defined by letting $a\le b\Leftrightarrow\forall w\in W\ a(w)\le b(w)\Leftrightarrow a\wedge b=a$.
The operator $ \Necessarly $ satisfies the following properties concerning this partial order:
\begin{itemize}
\item $a\le b\Rightarrow \Necessarly a\le \Necessarly b$
\item $ \mathcal E \subseteq \mathcal E'\Rightarrow\forall a\in \set{0,1}^W\ \Necessarly_{(W, \mathcal E')}a\le \Necessarly_{(W, \mathcal E )}a$
\item For any $b$ in the range of $ \Necessarly $, there is a $\le $-least $a$ such that $ \Necessarly a=b$, namely $a=\bigwedge_{a'\in \Necessarly^{-1}( \set{b} )}a'$.
\end{itemize}
\end{remark}

\begin{definition}
Let $ \Kripke =(W, \mathcal E )$ be a Kripke frame.
\begin{itemize}
\item A \emph{cycle} in $ \Kripke $ is a finite subframe $( \set{x_0,\ldots ,x_n} , \mathcal E')$ such that $x_0 \mathcal E'\ldots \mathcal E'x_n \mathcal E'x_0$, and the relation $ \mathcal E'$ does not hold for any other pair of elements of $ \set{x_0,\ldots ,x_n} $ (notice that for $n=0$ this means $x_0 \mathcal E x_0$, i.e., every loop is a cycle).
\item A \emph{line} in $ \Kripke $ is a subframe $(\{ x_i\}_{i\in \Z }, \mathcal E')$ such that $\forall i,j\in \Z \ (x_i \mathcal E'x_j\Leftrightarrow j=i+1)$ (in particular, lines are infinite and do not contain cycles).
\end{itemize}
\end{definition}

We point out that this definition of a cycle in a frame $ \Kripke $ is more restrictive than the usual definition of cycles for directed graphs: by asking that the cycle is a subframe, we require that the only edges in $ \Kripke $ between the elements of the cycle are the edges in the cycle itself.

We are now able to show that $ \Necessarly $ is an isomorphism if and only if the Kripke frame is a disjoint union of its cycles and lines.

\begin{theorem} \label{mainthm}
Let $\{ (W_i, \mathcal E_i)\}_{i\in I}$ be the collection of all cycles and lines of the Kripke frame $ \Kripke =(W, \mathcal E )$.
Then modal operator $ \Necessarly : \set{0,1}^W\to \set{0,1}^W$ is an isomorphism if and only if:
\begin{itemize}
\item $W=\bigcup_{i\in I}W_i$ and this is a disjoint union; and
\item $ \mathcal E =\bigcup_{i\in I} \mathcal E_i$ and this is a disjoint union.
\end{itemize}
\end{theorem}

\begin{proof}
Assume first that the condition on the Kripke frame holds.
Then for every $w\in W$ there is exactly one element $S(w)\in W$ such that $w \mathcal E S(w)$; similarly, there is exactly one element $P(w)\in W$ such that $P(w) \mathcal E w$, and functions $S,P:W\to W$ are bijections such that $P=S^{-1}$.
So, $\forall a\in \set{0,1}^W\ \forall w\in W\ \Necessarly a(w)=aS(w)$.
Consequently, given any $b\in \set{0,1}^W$ one has $\forall w\in W\ b(w)= \Necessarly (bP)(w)$, showing that $b= \Necessarly (bP)$ and that $ \Necessarly $ is surjective.
On the other hand, let $a,a'\in \set{0,1}^W$ be such that $a(w)\neq a'(w)$ for some $w\in W$; then $ \Necessarly a(P(w))=a(w)\neq a'(w)= \Necessarly a'(P(w))$, establishing the injectivity of $ \Necessarly $.

Conversely, assume that $ \Necessarly $ is bijective.

First notice that given any $w\in W$ there must be some $w'\in W$ with $w \mathcal E w'$: otherwise for any $a\in \set{0,1}^W$ one would have $ \Necessarly a(w)=1$, contradicting the surjectivity of $ \Necessarly $.
We claim now that for every $w\in W$ there is $y\in W$ such that $w \mathcal E y$ and for no $z\neq w$ one has $z \mathcal E y$.
Otherwise, if $w$ is such that every time $w \mathcal E y$ there is $z\neq w$ such that $z \mathcal E y$, given $a$ with $ \Necessarly a(w)=0$ there would exist $z\neq w$ such that $ \Necessarly a(z)=0$.
But then the function taking value $0$ in $w$ and $1$ elsewhere would not be in the range of $ \Necessarly $, reaching a contradiction.
So let $S:W\to W$ be a function assigning to each $w$ an element $y$ as above.

Analogously, given any $w\in W$ there exists $w'\in W$ such that $w' \mathcal E w$: otherwise if $a,a'\in \set{0,1}^W$ agree everywhere except on $w$, then $ \Necessarly a= \Necessarly a'$, against the injectivity of $ \Necessarly $.
Moreover, for every $w\in W$ there exists $y\in W$ such that $y \mathcal E w$ and for no $z\neq w$ one has $y \mathcal E z$.
Indeed, if $w$ were such that each time $y \mathcal E w$ there exists $z\neq w$ with $y \mathcal E z$, let $a,a'\in \set{0,1}^W$ be such that:
\begin{itemize}
\item $a(z)=0$ whenever there is $y\in W$ such that $y \mathcal E w,y \mathcal E z$ both hold (in particular, $a(w)=0$);
\item $a'$ agrees with $a$ on $W\setminus \set{w} $, but $a'(w)=1$.
\end{itemize}
Then $ \Necessarly a= \Necessarly a'$, contradicting the fact that $ \Necessarly $ is injective.
This allows to define a function $P:W\to W$ assigning to every $w$ an element $y$ as above.

Notice now that, for all $w\in W$, one has both $PS(w)=w$ and $SP(w)=w$, that is $P=S^{-1}$.
This implies that for every $w\in W$ there is a unique $y\in W$ such that $w \mathcal E y$, namely $y=S(w)$; similarly, there is a unique $z\in W$ such that $z \mathcal E w$, namely $z=P(w)$.
So the desired decomposition of $ \Kripke $ into cycles and lines follows.
\end{proof}

As a consequence, on a finite frame, operator $ \Necessarly $ is an isomorphism if and only if the frame is the disjoint union of its cycles.

\begin{corollary} \label{injsurj}
Let $ \Kripke =(W, \mathcal E )$ be a finite Kripke frame, and let $\{ (W_i, \mathcal E_i)\}_{i\in I}$ be the collection of all cycles of $ \Kripke $.
Then the modal operator $\Necessarly : \set{0,1}^W\to \set{0,1}^W$ is injective if and only if it is surjective, if and only if $\{ W_i\}_{i\in I}$ is a partition of $W$ and $\{ \mathcal E_i\}_{i\in I}$ is a partition of $ \mathcal E $.
\end{corollary}

\begin{proof}
The first equivalence holds as $ \set{0,1}^W$ is finite.
As for the second one, use Theorem \ref{mainthm} and the observation that every line is infinite.
\end{proof}

Notice that in the proof of the forward implication in Theorem \ref{mainthm}, to unveil the structure of the Kripke frame we did not use the full hypothesis of bijectivity of $ \Necessarly $, but an apparently weaker condition.
The reason is contained in the following fact.

\begin{proposition} \label{newlemma}
\begin{enumerate}
\item If every element of $ \set{0,1}^W$ assuming exactly once value $0$ is in the range of $ \Necessarly $, then $ \Necessarly : \set{0,1}^W\to \set{0,1}^W$ is surjective.
\item If $ \Necessarly a\ne \Necessarly a'$ for every $a,a'\in \set{0,1}^W$ differing on exactly one value, then $ \Necessarly $ is injective.
\end{enumerate}
\end{proposition}

\begin{proof}
(1) Let $b\in \set{0,1}^W$.
Then $b=\bigwedge_{b(w)=0}b_w$, where
\[ \left \{ \begin{array}{l}
b_w(w)=0 \\
b_w(w')=1 \text{ for } w'\in W\setminus\{ w\}
\end{array} \right . .
\]
By the hypothesis, for each $w\in W$ let $a_w\in \set{0,1}^W$ such that $ \Necessarly a_w=b_w$.
Then, by Eq.~\eqref{neweqn},
\[
b=\bigwedge_{b(w)=0} \Necessarly a_w= \Necessarly (\bigwedge_{b(w)=0}a_w)\in \mathrm{range} ( \Necessarly ) \, .
\]

(2) Suppose towards a contradiction that $a_1,a_2$ are distinct and $ \Necessarly a_1= \Necessarly a_2$.
Let $a=a_1\wedge a_2$, so that by Proposition \ref{boxmorphism} also $ \Necessarly a= \Necessarly a_1$.
Since $a_1\ne a_2$, there exist $i\in \set{1,2} ,w\in W$ such that $a(w)=0,a_i(w)=1$.
Define $a'\in \set{0,1}^W$ by letting
\[ \left \{ \begin{array}{l}
a'(w)=1 \\
a'(w')=a(w') \text{ for } w'\in W\setminus\{ w\}
\end{array} \right . .
\]
Since $a\le a'\le a_i$, then $ \Necessarly a= \Necessarly a'$ by Remark \ref{movedrmk}, but this is a contradiction as $a,a'$ differ on exactly one argument.
\end{proof}

\begin{corollary} \label{newthm}
Let $ \Kripke $ be a finite Kripke frame.
Then the following are equivalent:
\begin{enumerate}
\item Every $b\in \set{0,1}^W$ assuming exactly once value $0$ is in the range of $ \Necessarly $.
\item If $a,a'\in \set{0,1} $ differ for exactly one value, then $ \Necessarly a\ne \Necessarly a'$.
\item $ \Necessarly $ is an isomorphism.
\end{enumerate}
\end{corollary}

Notice that for infinite Kripke frames the condition of surjectivity (or injectivity) of $ \Necessarly $ alone does not imply that $ \Necessarly $ is an isomorphism.
Let indeed $ \Kripke $ be the set $ \N $ of natural numbers endowed with the relation $ \mathcal E $ defined by
\[
n \mathcal E m\Leftrightarrow m=n+1.
\]
Then $ \Necessarly $ is surjective, since for any $b\in \set{0,1}^W$ we have $b= \Necessarly a$ where $a$ is defined by letting $a(0)$ be arbitrary and $a(h)=b(h-1)$ for $h>0$, but $ \Necessarly $ is not injective since every element of $ \set{0,1}^W$ has two preimages.

Choosing instead $ \Kripke'=( \N , \mathcal E')$ where
\[
n \mathcal E'm\Leftrightarrow n=m+1
\]
one obtains that $ \Necessarly $ is injective, since if $a(n)\ne a'(n)$ then $ \Necessarly a(n+1)\ne \Necessarly a'(n+1)$, but not surjective, as $ \Necessarly a(0)=1$ for every $a$.

In the next section we present an algorithmic way to describe the range of $ \Necessarly $ through systems of binomial equations, assuming that the Kripke frame is finite.

\section{An application of toric ideals}
\label{sec:commutative-algebra}
Let a finite Kripke frame $\Kripke = (W,\mathcal E)$ be given, where we can assume that $W= \set{1,\ldots ,K} $. The adjacency matrix of $\mathcal K $ is denoted by $E$ and $e_w$ is the $w$-th row of $E$.
Since we deal with functions $a:W\to \set{0,1}$, so with elements of $ \set{0,1}^K$, the range of $ \Necessarly $ is a subset of $ \set{0,1}^K$. We want to obtain equations for $ \mathrm{range} ( \Necessarly )$ as a subvariety of $ \C^K$ --- the use of the field $ \C $ allowing us to apply well established results in Commutative Algebra.

Recall that, given an ideal $I$ in the polynomial ring $ \C [x_1,\ldots ,x_n]$, the variety of $I$ is the set
\[
V(I)= \setof{a\in \C^n}{\forall f\in I\ f(a)=0} \ .
\]
Conversely, for $A\subseteq \C^n$, the ideal of $A$ is
\[
\Ideal (A)= \setof{f\in \C [x_1,\ldots ,x_n]}{\forall a\in A\ f(a)=0} \ .
\]
Since every ideal is finitely generated we write $I=\langle f_1,\ldots ,f_r\rangle $ for the ideal generated by the polynomials $f_1,\ldots ,f_r$.

From the definition of the modal operator in Eq.~\eqref{eq:square} we see that each value $\Necessarly a(w)$ has the algebraic form of a square-free monomial in the indeterminates $a(w')$, $w' \in W$. We are in the special case where the value of each indeterminate is either 0 or 1.
We thus consider two sets of indeterminates:
\begin{itemize}
\item $t_w = a(w)$, $w \in W$;

\item $z_w = \Necessarly a(w)$, $w \in W$,
\end{itemize}
and work in the polynomial ring $ \C [t_w,z_w:w \in W]$.

Since $a(w) \in \{0,1\}$ for all $w$, we define a first set of equations and the corresponding ideal

\begin{equation*} 
t_w^2-t_w=0, \quad \text{for } w \in W, \qquad \qquad {\mathcal I}_L= \langle t_w^2-t_w  :  w \in W\rangle .
\end{equation*}

From Eq.~\eqref{eq:square}, we define a second set of equations involving the $z$'s and the corresponding ideal

\begin{equation} \label{toric:id}
z_w-\prod_{w' \in N(w)} t_{w'}=0, \quad \text{for } w \in W, \qquad \qquad {\mathcal I}_T= \left \langle z_w-\prod_{w' \in N(w)} t_{w'}  :  w \in W \right \rangle .
\end{equation}
The ideal ${\mathcal I}_T$ in Eq.~\eqref{toric:id} is a toric ideal in the indeterminates $z_w$, $w \in W$. Toric ideals are special binomial ideals, see e.g. \cite[Ch. 4]{sturmfels:1996} for a general reference on toric ideals. They are applied in many contexts, and especially in Algebraic Statistics for contingency tables, to describe varieties (i.e., statistical models) for finite sample spaces, see e.g. \cite{rapallo:2007}.

Now, define the ideal
\begin{equation*} 
{\mathcal I} = {\mathcal I}_L + {\mathcal I}_T
\end{equation*}
and consider the affine space $ \C^{2K} = \C^{K}_{(t)} \times \C^{K}_{(z)}$.

So, in the space $ \C^{2K}$ we can define the varieties $V({\mathcal I}_L)$, $ V({\mathcal I}_T)$, and $V({\mathcal I })$.
While the variety $V({\mathcal I}_L)$ is clearly the set of all points whose $t$-coordinates are $0$ or $1$, the other two varieties are more interesting. In particular, note that the variety $V({\mathcal I}_T)$ is the toric variety of the adjacency matrix $E$ of the Kripke frame.

The projections of such varieties onto the affine space $\C^{K}_{(z)}$ are denoted with
\[
\widetilde V({\mathcal I}_L) = \pr_{ \C^{K}_{(z)}} V({\mathcal I}_L) \qquad \widetilde V({\mathcal I}_T) = \pr_{ \C^{K}_{(z)}} V({\mathcal I}_T) \qquad \widetilde V({\mathcal I}) = \pr_{ \C^{K}_{(z)}} V({\mathcal I}) \, .
\]
Notice that $ \widetilde{V} ( \mathcal I )= \mathrm{range} ( \Necessarly )$.  

On the other hand, let the elimination ideals of the $t$'s indeterminates be:
\begin{equation*}
\widetilde{\mathcal I_L} = \elimof{(t_w)_{w \in W}}{\mathcal I_L} \qquad \widetilde{\mathcal I_T} = \elimof{(t_w)_{w \in W}}{\mathcal I_T} \qquad \widetilde{\mathcal I} = \elimof{(t_w)_{w \in W}}{\mathcal I} \, .
\end{equation*}
It is known, see e.g. \cite[Theorem 3, page 131]{cox|little|oshea:1997}, that the varieties of such elimination ideals are the Zariski closure of the above projections.
Since every finite set is Zariski closed, we can conclude that
\[
\mathrm{range} ( \Necessarly ) = \widetilde{V} ( \mathcal I )=V( \widetilde { \mathcal I } )\ .
\]
Consequently, any set of generators of $ \widetilde{ \mathcal I }$ provides a system of equations for $ \mathrm{range} ( \Necessarly )$.
Moreover, notice that the ideal $\widetilde{\mathcal I}$ is both an elimination ideal and a binomial ideal (see e.g. \cite{eisenbud1996}). Thus, a set of generators of such an ideal can be computed through Gr\"obner bases with symbolic software (in the examples of next section we have used CoCoA, see \cite{cocoa}).

For any $\alpha \in \Z^K$, let $\alpha_+,\alpha_-\in \N^K$ have disjoint support and be such that $\alpha =\alpha_+-\alpha_-$.
The following theorem uses the theory of toric ideals of Sturmfels, see \cite[Ch. 4]{sturmfels:1996}, and it describes the generators of the ideals $\widetilde{\mathcal I_T}$ and $\widetilde{\mathcal I}$, thus providing the announced equations for $ \mathrm{range} ( \Necessarly )$ depending on the incidence matrix $E$.
Recall that, given $\beta\in \N^K$, a compact expression like $z^{\beta }$ denotes the product $\prod_{i=1}^Kz_i^{\beta (i)}$.

\begin{theorem} \label{mainproposition}
  \begin{enumerate}
\item  The ideal $ \widetilde{ \mathcal I_T} $ is generated by the binomials
  \begin{equation*}
    z^{\alpha_+} - z^{\alpha_-}, \quad \alpha\in \Z^K\cap \kerof{E^t} \, ;
  \end{equation*}
\item The ideal $\widetilde{\mathcal I}$ is generated by the binomials
\[
z_w^2 -z_w \ ,\ w \in W
\]
and by the square-free binomials of the form $z^u - z^v$ with $u,v \in \{0,1\}^K$ such that $\mathrm{supp}(E^tu)=\mathrm{supp}(E^tv)$.
\end{enumerate}
\end{theorem}

\begin{proof}
(1) This is, for instance, \cite[Lemma 1.1(a)]{sturmfarticle}.

(2) Since $\widetilde{\mathcal I}$ is a binomial ideal, we only need to find the generators of $\widetilde{\mathcal I}$ by looking at the binomials of $ \C [z]$ belonging to $\widetilde{\mathcal I}$.

Ideal $\widetilde{\mathcal I}$ contains the binomials
\[
z_w^2 -z_w \ ,\ w \in W\ .
\]
Moreover, $\widetilde{\mathcal I}$ contains a square-free binomial of the form $z^u - z^v$ with $u,v \in \{0,1\}^K$ if and only if $\mathrm{supp}(E^tu)=\mathrm{supp}(E^tv)$. In fact, in $ \C [z,t]$ modulo $ \widetilde{ \mathcal I }$ we have
\[
z^u-z^v = \prod_{w \ | \ u(w)=1} \prod_{w' \in N(w)} t_{w'} - \prod_{w \ | \ v(w)=1} \prod_{w' \in N(w)} t_{w'}
\]
\[
= \prod_{w'} t_{w'}^{\sum E(u=1,w')}- \prod_{w'} t_{w'}^{\sum E(v=1,w')}
\]
and this binomial belongs to ${\mathcal I}$ if and only if $\mathrm{supp}(E^tu)=\mathrm{supp}(E^tv)$.
\end{proof}

\begin{remark} \label{rmkenumerate}
\begin{enumerate}
\item If $\alpha\in \Z^K\cap \kerof{E^t} $, letting $u,v\in \set{0,1}^K$ be defined by
\[
u(w)=\min (1,\alpha_+(w)), \quad v(w)=\min (1,\alpha_-(w))
\]
then $ \mathrm{supp} (E^tu)= \mathrm{supp} (E^tv)$; in other words, each of the binomials generating $ \widetilde{ \mathcal I_T}$ as for Theorem \ref{mainproposition}(1) gives rise to a binomial in the set of generators for $ \widetilde{ \mathcal I }$ described in Theorem \ref{mainproposition}(2).
In fact, $ \widetilde{ \mathcal I_T} \subseteq \widetilde{ \mathcal I } $.
However the binomials obtained in this way, together with the binomials $z_w^2-z_w$, are in general not enough to generate $ \widetilde{ \mathcal I }$: see, for instance, Examples \ref{exampleten} and \ref{exampletwelve} below.
We give in Proposition \ref{geomtame} a condition under which they suffice.
\item Theorem \ref{mainproposition}(1) says that all $\Necessarly a$, in addition to assuming values in $ \set{0,1} $, are subject to the following constraints:
\begin{equation} \label{generatingeqs}
\prod_{w\in W} (\Necessarly a(w))^{\alpha_+(w)} = \prod_{w\in W} (\Necessarly a(w))^{\alpha_-(w)}, \quad \alpha\in \Z^K\cap \kerof{E^t} \ .
\end{equation}
Similarly as what remarked above, the equations \eqref{generatingeqs}, together with the requirement for $ \Necessarly a$ of taking values in $ \set{0,1} $, are in general not enough to define the range of $ \Necessarly $: see Examples \ref{exampleten} and \ref{exampletwelve}.
\end{enumerate}
\end{remark}

The equations for $ \mathrm{range} ( \Necessarly )$ provided by Theorem \ref{mainproposition} are usually quite complicated.
We now look for some conditions on the Kripke frame under which one can describe $ \mathrm{range} ( \Necessarly )$ with a simplified set of equations.
A reasonable simplification (see Remark \ref{rmkenumerate}(2)) would be to describe $ \mathrm{range} ( \Necessarly )$ just using the equations stating that $ \Necessarly$ takes values in $ \set{0,1}^W$ and \eqref{generatingeqs}; in other words, using (the generators of) the ideal $\langle z_w^2-z_w:w\in W\rangle + \widetilde{ \mathcal I_T} $.
This is captured by the following definition.

\begin{definition} \label{deftame}
Let $ \Kripke =(W, \mathcal E )$ be a finite Kripke frame.
We say that $ \Kripke $ is \emph{tame} if
\[
\mathrm{range} ( \Necessarly )=V(\langle z_w^2-z_w:w\in W\rangle + \widetilde{ \mathcal I_T} ).
\]
\end{definition}

To shorten notations, we set
\[
J=\langle z_w^2-z_w:w\in W\rangle + \widetilde{ \mathcal I_T} ;
\]
in particular, $J\subseteq \widetilde{ \mathcal I }$ and thus $ \mathrm{range} ( \Necessarly )=V( \widetilde{ \mathcal I })\subseteq V(J)$.

Notice that if $ \Kripke $ is such that $J= \widetilde { \mathcal I } $, then it is tame.
The converse holds as well.
To see this we start with the following general lemma.
It could be well-known, but we could not find it in the literature.

\begin{lemma} \label{lemradnew}
Let $I$ be an ideal in $ \C [z_1,\ldots ,z_K]$ be such that $z_i^2-z_i\in I$ for every $i\in\{ 1,\ldots ,K\} $.
Then $I$ is a radical ideal.
\end{lemma}

\begin{proof}
First, $I$ is a $0$-dimensional ideal, as $V(I)\subseteq \set{0,1}^K$, so one can apply Seidenberg's algorithm (see \cite{seiden1974}) to compute $ \sqrt{I} $.
Namely, let $I_i=I\cap \C [z_i]$; then, as an ideal in $ \C [z_i]$, it turns out that $I_i$ is generated by a single monic polynomial, say $f_i$.
Notice that all $f_i$ can be assumed to be non-constant: otherwise, as $f_i\in I$, it follows that $I= \C [z_1,\ldots ,z_K]$, which is radical.

If $g_i$ denotes the squarefree part of $f_i$, then
\[
\sqrt{I} =I+\langle g_1,\ldots ,g_n\rangle .
\]
Consequently, it is enough to show that $g_i\in I$ for every $i\in \set{1,\ldots ,K} $.
Moreover, $g_i= \frac{f_i}{h_i} $, where $h_i$ is the monic polynomial that is the greatest common divisor of $f_i,f_i'$.

Since $z_i^2-z_i\in I_i$, polynomial $f_i$ must divide $z_i^2-z_i$, so there are various possibilities.
\begin{itemize}
\item $f_i=z_i^2-z_i$.
Then $f_i'=2z_i-1$, so $h_i=1$ and $g_i=f_i\in I$.
\item $f_i=z_i-1$.
Then $f_i'=1$, so $h_i=1$ and $g_i=f_i\in I$.
\item $f_i=z_i$.
Then $f_i'=1$, so $h_i=1$ and again $g_i=f_i\in I$.
\end{itemize}
\end{proof}

Lemma \ref{lemradnew} applies in particular to ideal $J$.

\begin{proposition} \label{propatsat}
Suppose that $ \Kripke $ is tame.
Then $J= \widetilde{ \mathcal I } $.
\end{proposition}

\begin{proof}
Since $J\subseteq \widetilde{ \mathcal I }$, it is enough to observe that
\[
\widetilde{ \mathcal I } \subseteq \sqrt{ \widetilde{ \mathcal I }}=\Ideal (V( \widetilde{ \mathcal I } ))= \Ideal(V (J))= \sqrt{J} =J\ .
\]
where the last equality holds by Lemma \ref{lemradnew}.
\end{proof}

We now introduce a class of Kripke frames that turn out to be tame: these are the Kripke frames with the property that any two neighborhoods (see section \ref{boxmorph}) are either disjoint or they coincide, that is such that the neighborhoods $N(w)$ cut $\bigcup_{w\in W}N(w)$ into a partition.
They are described by the following definition.

\begin{definition} \label{defcut}
We say that a Kripke frame $ \Kripke $ is a \emph{partitioning frame} if
\begin{equation} \label{eqpartition}
\forall w,w'\in W\ (N(w)\cap N(w')\neq\emptyset\Rightarrow N(w)=N(w'))\ .
\end{equation}
\end{definition}

  \begin{example}[Partitioning frames]\
The following are examples of partitioning frames.
     \begin{description}
\item[Complete bipartite graphs] Complete bipartite graphs are partitioning frames.
All nodes in the graph are assigned either color $0$ or $1$, and any node colored $0$ is a neighbor of every node of color $1$ and viceversa.
      \item[Isomorphic $\Necessarly$] Examples of partitioning frames are those for which $ \Necessarly $ is an isomorphism, see Theorem~\ref{mainthm}. Notice that for a finite such frame, $ \widetilde{ \mathcal I_T}$ is the null ideal by Theorem \ref{mainproposition}, as $E$ is non-singular.
Moreover, $ \widetilde{ \mathcal I }=\langle z_w^2-z_w:w\in W\rangle $, since in every row and every column of $E$ there is exactly one non-null term, so given $u,v\in \set{0,1}^W$ it holds that $ \mathrm {supp} (E^tu)= \mathrm{supp} (E^tv)\Leftrightarrow u=v$.
Thus this shows directly that these Kripke frames are tame.
\item[Equivalence relations] All $ \Kripke =(W, \mathcal E )$ with $ \mathcal E $ an equivalence relation are partitioning frames: this is the class of Kripke frames defined by epistemic logic $S5$, that is the logic characterized by the axioms $ \Necessarly p\rightarrow p$ and $ \Possibly p\rightarrow \Necessarly \Possibly p$, see e.g. \cite[Ch. 4]{blackburn}.
\item[Trees] Fix a graph-theoretic tree, and choose a root $\bar w$.
The set of nodes of the tree is partitioned according to the distance of each element $w$ from the root, that is the length of the unique path from $\bar w$ to $w$. A node $w$ at distance $d$ is connected to a single node at distance $d-1$ (if $d>0$) and a set of nodes $L(w)$ at distance $d+1$, $\{ L(w)\mid d(w) =d\} $ being a partition of the nodes at distance $d+1$. It follows that in general undirected trees are not partitioning frames, while directed rooted trees and directed rooted trees with inversed arrows are partitioning frames.
\end{description}
\end{example}

\begin{proposition} \label{geomtame}
If $ \Kripke $ is a finite partitioning frame, then $ \Kripke $ is tame.
\end{proposition}

\begin{proof}
It is enough to show that $V(J)\subseteq V( \widetilde{ \mathcal I })$.
Assume that $b\in V(J)$.
This implies that $b(w)\in \set{0,1} $ for all $w\in W$, and that $b(w_0)=b(w_1)$ whenever $N(w_0)=N(w_1)$.
So define $a\in \set{0,1}^K$ by letting $a(w')=b(w)$, for any $w$ such that $w'\in N(w)$, and defining $a(w')$ arbitrarily if $w'\in W\setminus\bigcup_{w\in W}N(w)$.
Then $b= \Necessarly a\in \mathrm{range} ( \Necessarly )=V ( \widetilde{ \mathcal I } )$.
\end{proof}

Observe that being a partitioning frame is not a necessary condition for being tame: see Examples \ref{aftercorollary} and \ref{exampleleven}.

Finally, let us point out that in this paper we focused on $ \Necessarly $, but the operator $\Possibly$ can be described with the same technique. In fact,
\begin{equation*}
\Possibly a_w =1 - \prod_{w' \in N(w)} (1 - t_{w'})
\end{equation*}
and one can use the theory above through the following substitutions:
\[
t'_w \mapsto 1-t_w \qquad  z'_w \mapsto 1-z_w\ .
\]
Notice that the substitution $t'_w \mapsto 1-t_w$ together with $t_w^2-t_w=0$ implies that $(t'_w)^2-t'_w=0$.

\section{Some explicit examples} \label{examples}
As a first example, we describe a tame Kripke frame that is not a partitioning frame.

\begin{example} \label{aftercorollary}
Let $W= \set{1,2,3} $, with incidence matrix
\[
E= \bordermatrix[{[}{]}]{& 1 & 2 & 3 \cr
1 & 1 & 1 & 0 \cr
2 & 0 & 1 & 1 \cr
3 & 0 & 0 & 0} \ .
\]
Then $ \widetilde{ \mathcal I }=J=\langle z_3-1,z_1^2-z_1,z_2^2-z_2,z_3^2-z_3\rangle =\langle z_3-1,z_1^2-z_1,z_2^2-z_2\rangle $.
\end{example}

The following example of a partitioning frame provides a simple illustration of the procedure discussed in the previous section.

\begin{example}[The symmetric $4$-cycle] \label{byhands}
  Consider the Kripke frame $\Kripke = ( \set{1,2,3,4}, \mathcal E)$, with adjacency matrix
  \begin{equation*}
E = \bordermatrix[{[}{]}]{& 1 & 2 & 3 & 4 \cr
1 & 0 & 1 & 0 & 1 \cr
2 & 1 & 0 & 1 & 0 \cr
3 & 0 & 1 & 0 & 1 \cr
4 & 1 & 0 & 1 & 0} \ .
\end{equation*}
Notice that this Kripke frame is not a disjoint union of cycles, so by Corollary \ref{injsurj} the range of $ \Necessarly a$ is a proper subvariety of $ \set{0,1}^4$.

Eq.~\eqref{eq:modalistoric} becomes
\begin{equation*}
\left\{  \begin{aligned}
    \Necessarly a(1) &= a(2)a(4) \\
   \Necessarly a(2) &= a(1)a(3) \\
   \Necessarly a(3) &= a(2)a(4) \\
   \Necessarly a(4) &= a(1)a(3) \\
  \end{aligned} \right. \ .
\end{equation*}

Let $\gamma = (\gamma(1),\gamma(2),\gamma(3),\gamma(4))$ be a non-zero integer vector such that $E^t\gamma = 0$, so that
\begin{equation*}
  \left\{
    \begin{aligned}
      0 &= \gamma(2)+\gamma(4) \\
      0 &= \gamma(1)+\gamma(3) \\
    \end{aligned}
\right. \ .
\end{equation*}

The solutions are all vectors of the form $(u,v,-u,-v)$ for $u,v\in \Z $.
The vectors
\begin{equation*}
(1,0,-1,0),(0,1,0,-1)
\end{equation*}
generate with integer coefficients all such solutions, and have disjoint supports.
These vectors can be split as
\begin{equation*}
  (1,0,0,0) - (0,0,1,0) = (1,0,-1,0),\quad (0,1,0,0)-(0,0,0,1)=(0,1,0,-1)\ .
\end{equation*}
So the ideal $ \widetilde { \mathcal I_T} $ is generated by the binomials
\begin{equation*} 
  z_1-z_3,\quad z_2-z_4.
\end{equation*}

To generate $ \widetilde{ \mathcal I } $, in this case it is enough to add the binomials $z_w^2-z_w$, since $ \Kripke $ is a partitioning frame, so it is tame.
In conclusion the range of the necessitation operator $\Necessarly$ consists of the 4 points
\begin{equation*}
  (0,0,0,0), (1,0,1,0), (0,1,0,1), (1,1,1,1)\ .
\end{equation*}
\end{example}

Since the toric ideal $ \widetilde{ \mathcal I_T} $ associated to the adjacency matrix $E$ is a subset of the relevant binomial ideal $\widetilde{\mathcal I}$, this implies that a subset of the generators can be computed through specialized software for toric ideals (for instance, with {\tt 4ti2}, see \cite{4ti2}). Such a computation exploits the special structure of toric ideals and therefore makes it possible some computations also for large frames, where the elimination technique fails.

When the equality $ \widetilde{ \mathcal I }=J$ fails, that is when the Kripke frame is not tame, the computation of $ \widetilde{ \mathcal I }$ is more complex.
We present here some examples where the computation of a Gr\"obner basis for such an ideal has been carried out in {C}o{C}o{A}, \cite{cocoa}.


\begin{example}[The reflexive symmetric 4-cycle] \label{exampleten}
Let us now consider the Kripke frame displayed in Fig.~\ref{figsym4cycle} together with its adjacency matrix.
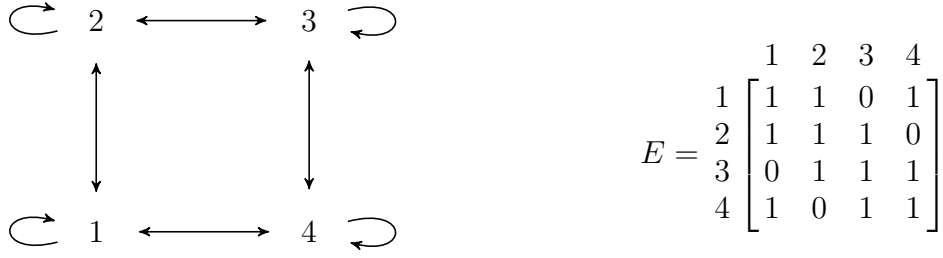
\begin{figure}
\begin{minipage}{0.48\textwidth}
\begin{center}
\begin{tikzpicture}[<->,>=stealth',shorten >=1pt,auto,node distance=2.8cm,
                    semithick,scale=0.8]
  \tikzstyle{every state}=[fill=white,draw=none,text=black]

  \node[state] (n1)               {$1$};
  \node[state]         (n2) [above of=n1] {$2$};
  \node[state]         (n3) [right of=n2] {$3$};
  \node[state]         (n4) [below of=n3] {$4$};
  \path (n1) edge              node {} (n2)
             edge [loop left]  node {} (n1)
        (n2) edge              node {} (n3)
             edge [loop left]  node {} (n2)
        (n3) edge              node {} (n4)
             edge [loop right] node {} (n3)
        (n4) edge              node {} (n1)
             edge [loop right] node {} (n4);
\end{tikzpicture}
\end{center}
\end{minipage}
\begin{minipage}{0.48\textwidth}
\begin{center}
\begin{equation*}
E = \bordermatrix[{[}{]}]{& 1 & 2 & 3 & 4 \cr
1 & 1 & 1 & 0 & 1 \cr
2 & 1 & 1 & 1 & 0 \cr
3 & 0 & 1 & 1 & 1 \cr
4 & 1 & 0 & 1 & 1}
\end{equation*}
\end{center}
\end{minipage}
\caption{The reflexive symmetric $4$-cycle and its adjacency matrix.} \label{figsym4cycle}
\end{figure}

This is the reflexive version of Example \ref{byhands}. The corresponding elimination ideal is generated by the Gr\"obner basis given by:
\begin{equation*}
z_2z_3 - z_3z_4, \ -z_2z_4 + z_3z_4, \ -z_1z_3 + z_3z_4, \ -z_1z_4 + z_3z_4, \ -z_1z_2 + z_3z_4
\end{equation*}
plus the binomials $z_i^2-z_i$, $i=1,\ldots, 4$. In this example the toric ideal $\widetilde{ \mathcal I_T}$ is the null ideal, since the adjacency matrix is non-singular.
So this Kripke frame is not tame.
\end{example}

\begin{example}[The reflexive oriented $4$-cycle] \label{exampleleven}
We slightly modify the adjacency matrix above by choosing an orientation in the $4$-cycle. The Kripke frame and the adjacency matrix are displayed in Fig.~\ref{figor4cycle}.
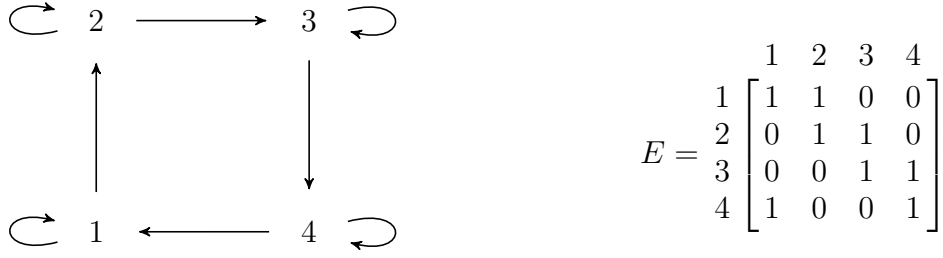
\begin{figure}
\begin{minipage}{0.48\textwidth}
\begin{center}
\begin{tikzpicture}[->,>=stealth',shorten >=1pt,auto,node distance=2.8cm,
                    semithick,scale=0.8]
  \tikzstyle{every state}=[fill=white,draw=none,text=black]

  \node[state] (n1)               {$1$};
  \node[state]         (n2) [above of=n1] {$2$};
  \node[state]         (n3) [right of=n2] {$3$};
  \node[state]         (n4) [below of=n3] {$4$};
  \path (n1) edge              node {} (n2)
             edge [loop left]  node {} (n1)
        (n2) edge              node {} (n3)
             edge [loop left]  node {} (n2)
        (n3) edge              node {} (n4)
             edge [loop right] node {} (n3)
        (n4) edge              node {} (n1)
             edge [loop right] node {} (n4);
\end{tikzpicture}
\end{center}
\end{minipage}
\begin{minipage}{0.48\textwidth}
\begin{center}
  \begin{equation*}
E = \bordermatrix[{[}{]}]{& 1 & 2 & 3 & 4 \cr
1 & 1 & 1 & 0 & 0 \cr
2 & 0 & 1 & 1 & 0 \cr
3 & 0 & 0 & 1 & 1 \cr
4 & 1 & 0 & 0 & 1}
\end{equation*}
\end{center}
\end{minipage}
\caption{The reflexive oriented $4$-cycle and its adjacency matrix.} \label{figor4cycle}
\end{figure}

In this case the elimination ideal $ \widetilde{ \mathcal I }$ is generated by the Gr\"obner basis given by:
\begin{equation*}
z_1z_3 - z_2z_4, \  -z_2z_3z_4 + z_2z_4, \ z_1z_2z_4 - z_2z_4
\end{equation*}
plus the binomials $z_i^2-z_i$, $i=1,\ldots, 4$. Here, the binomial $z_1z_3 - z_2z_4$ belongs to the toric ideal $\widetilde{ \mathcal I_T}$ and the toric ideal is actually a principal ideal generated by this binomial.

In fact $ \widetilde{ \mathcal I } =J$, since the polynomials $-z_2z_3z_4 + z_2z_4,z_1z_2z_4 - z_2z_4$ are redundant generators, as they belong to
\[
J=\langle z_1^2-z_1,z_2^2-z_2,z_3^2-z_3,z_4^2-z_4,z_1z_3-z_2z_4\rangle
\]
Thus this Kripke frame is tame.
\end{example}

\begin{example} \label{exampletwelve}
Let us consider the reflexive frame displayed in Fig.~\ref{bigfig}.
In this tree-like structure, the value of $\Necessarly a$ at a given world depends on the value of $a$ at the worlds that come from the same parent or are immediate descendants.
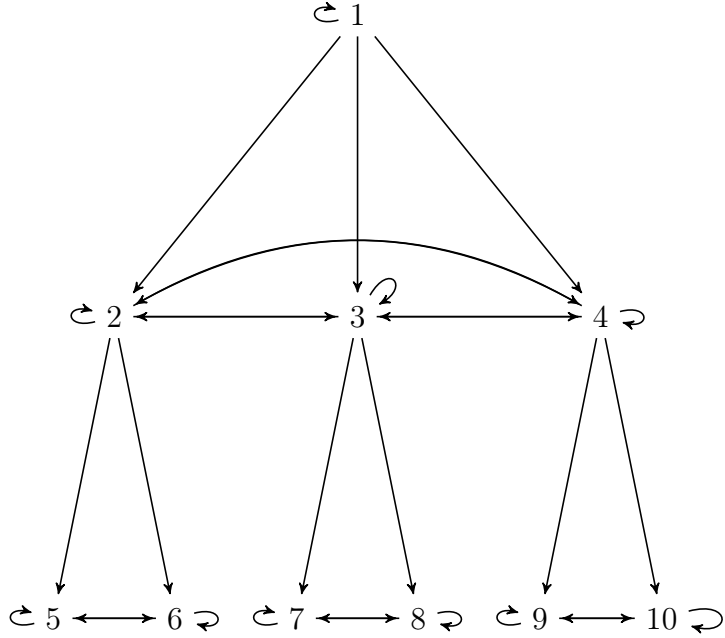
\begin{figure}
\begin{center}
\begin{tikzpicture}[->,>=stealth',shorten >=1pt,auto,node distance=2.8cm,
                    semithick,scale=0.8]
  \tikzstyle{every state}=[fill=white,draw=none,text=black]

   \node (n1) at (6,10) {1};
  \node (n2) at (2,5) {2};
  \node (n3) at (6,5) {3};
  \node (n4) at (10,5) {4};
   \node (n5) at (1,0) {5};
   \node (n6) at (3,0) {6};
   \node (n7) at (5,0) {7};
   \node (n8) at (7,0) {8};
   \node (n9) at (9,0) {9};
  \node (n10) at (11,0) {10};
  \path (n1) edge              node {} (n2)
             edge              node {} (n3)
             edge              node {} (n4)
             edge [loop left]  node {} (n1)
        (n2) edge              node {} (n3)
             edge [bend left]  node {} (n4)
             edge              node {} (n5)
             edge              node {} (n6)
             edge [loop left]  node {} (n2)
        (n3) edge              node {} (n2)
             edge              node {} (n4)
             edge              node {} (n7)
             edge              node {} (n8)
             edge [in=30,out=60,loop] node {} (n3)
        (n4) edge              node {} (n3)
             edge [bend right] node {} (n2)
             edge              node {} (n9)
             edge              node {} (n10)
             edge [loop right] node {} (n4)
        (n5) edge              node {} (n6)
             edge [loop left]  node {} (n5)
        (n6) edge              node {} (n5)
             edge [loop right] node {} (n6)
        (n7) edge              node {} (n8)
             edge [loop left]  node {} (n7)
        (n8) edge              node {} (n7)
             edge [loop right] node {} (n8)
        (n9) edge              node {} (n10)
             edge [loop left]  node {} (n9)
       (n10) edge              node {} (n9)
             edge [loop right] node {} (n10);
\end{tikzpicture}
\end{center}
\caption{The frame for Ex.~\ref{exampletwelve}} \label{bigfig}
\end{figure}
Apart from the binomials $z_i^2-z_i$, $i=1,\ldots, 10$, the binomial ideal $\widetilde { \mathcal I }$ is generated by 7 binomials:
\begin{itemize}
\item 5 linear binomials, the generators of the toric ideal $\widetilde{ \mathcal I_T}$:
\[
z_9 - z_{10}, \ z_7 - z_8,\  z_5 - z_6, \ z_3 - z_4,\  z_2 - z_4;
\]
\item 2 further reducible binomials not belonging to the toric ideal:
\[
-z_1z_6z_8z_{10} + z_4z_6z_8z_{10}, \ z_1z_4 - z_4 \, .
\]
\end{itemize}
\end{example}

\section{Final remarks and some questions} \label{fraq}

One of the original motivations that led C.S.~Lewis to the study of modal logic was philosophical. Namely, he was interested in finding a stronger definition of logical implication. He came out with various definitions among which the most popular today is the following: ``$p$ strongly implies $q$'' means $\Necessarly(p \rightarrow q)$.

The algebraic presentation of Kripke frame semantics we use in this paper provides a way to express such logical statements in polynomial algebra e.g., strong implication at world $w$ becomes the polynomial  $\prod_{w'\in N(w)} (1 - p(w') + p(w')q(w'))$. Computationally speaking, the algebraic presentation has some advantage with respect to the logical notation in that it opens up the opportunity to take advantage of another well-developed theory, Polynomial Commutative Algebra. In case of a finite Kripke frame, the relevant polynomial algebra enjoys a finiteness property because of the finite generation of polynomial ideals. Moreover, some of the algebraic ideals that are associated with the finite Kripke frame are of a special kind; namely, toric ideals. The combinatorial features of the theory of toric ideals produce a special type of computational algorithms that are currently implemented in symbolic software. While the computational complexity of such algorithms is very high, it is nevertheless useful to use such tools to study model examples. In this paper, we put together the two perspectives, Modal Logic and Computational Commutative Algebra. Though some interplay between the two subjects has already been exploited in the past (see for example \cite{chazarain1991}), as far as we know there is no much literature about this.

Thus the purpose of this paper can be seen as two-fold.
From the modal theoretic viewpoint, we think that it is interesting to develop computational tools describing the behaviour of objects studied in the field, in our case what kind of truth set a proposition of the form $ \Necessarly a$ can have.
On a more general level, although the study of the interplay between these two theories is in its early stages, the fact that several structures in algebraic statistics are described by toric ideals suggests that there may exist deeper connections between them that deserve to be investigated.

We propose here some basic questions that could help in starting such an investigation.
Among the early results obtained from the use of a Kripke frame $(W,\mathcal E)$ to define the semantics, there was a classification of different logical axiom systems according to the properties of the relation $\mathcal E$. Our approach is similar, in the sense we discuss the properties of the ``necessary proposition'' that is, the range of the operator $\Necessarly$, with the aid of a finite set of generators of a polynomial ideal and, in this way, we obtain properties of the Kripke frame.
As we applied the full force of the algebraic theory with its specialized notions of binomial ideal, toric ideal, radical ideal, and elimination ideal to obtain our results, two questions --- somehow dual to each other --- arise immediately.

\begin{question}
Can the results on the structure of $ \mathrm{range} ( \Necessarly )$ be obtained by purely logical means?
\end{question}

\begin{question}
Can the tools from polynomial algebras be applied to other problems coming from modal logic?
\end{question}

In section \ref{sec:commutative-algebra} we have shown that finite partitioning frames are tame, but a nice characterisation of tame frames is still lacking.
Moreover, while our notion of tameness seems to provide a reasonable simplification for the description of $ \mathrm{range} ( \Necessarly )$ via polynomial equations, there might be other ways to determine a simpler set of conditions than those provided by the full ideal $ \widetilde { \mathcal I } $.

\begin{question}
Is there a nice characterisation of tame frames?
\end{question}

\begin{question} \label{qst}
What are other classes of Kripke frames, different from the tame ones, that admit a simplified set of equations for $ \mathrm{range} ( \Necessarly )$?
\end{question}

Notice that to get our results, we relied on the combinatorial structure of the Kripke frames.
It would be interesting to find a semantical treatment.

\begin{question}
Is it possible to characterise tame Kripke frames, or Kripke frames that are simple in the sense of Question \ref{qst}, as those frames that satisfy a given set of axioms of modal logic?
\end{question}

We close this discussion with two general remarks.

As observed above, the techniques presented here rely on the finiteness of the Kripke frame $ \Kripke $ to produce the equations for the range of the modal operator $ \Necessarly $ using algorithms and tools from Algebraic Statistics.
We do not know whether the methods discussed in this paper can be adapted to yield useful information about infinite $ \Kripke $.

Finally, we point out that there are open questions also from the side of Algebraic Statistics. In Statistics, graphs like that in Ex.~\ref{byhands} represent graphical models where each node is a random variable and the edges account for the conditional independence statements among the variables. In the case of discrete random variables, the relevant probability models are toric varieties described by binomials (an example with four variables is extensively discussed in \cite[Ch. 6]{pistone|riccomagno|wynn:2001}).
Therefore a deeper investigation of the connections between the two fields could be fruitful from the statistical viewpoint as well, in particular with regard to Bayesian networks and causal inference.

\section*{Acknowledgements}
The authors would like to thank G. D'Agostino (University of Udine) and C. Bocci (University of Siena) for their criticism and comments on some earlier drafts of the paper.
R. Camerlo wishes to thank the \'Equipe de logique of the Universit\'e de Lausanne, where he was visiting while part of this research was carried out.
G. Pistone acknowledges the support of de Castro Statistics and of Collegio Carlo Alberto.
R. Camerlo is member of GNSAGA-INdAM; G. Pistone and F. Rapallo are members of GNAMPA-INdAM. This research has a financial support of the Universit\`a del Piemonte Orientale.

\end{document}